\documentclass[12pt]{article}
\usepackage{amsfonts,amsthm}
\usepackage{amssymb,amsmath}
\usepackage{enumerate}
\usepackage{amsmath}
\usepackage[dvipsnames]{xcolor}
\usepackage{hyperref}
\hypersetup{
    colorlinks=true,
    citecolor=blue,
    linkcolor=blue,
    filecolor=magenta,
    urlcolor=blue,
    pdftitle={Overleaf Example},
    pdfpagemode=FullScreen,
    }

\input{amssym.def}

\newtheorem{Definition}{Definition}[section]
\newtheorem{Theorem}[Definition]{Theorem}
\newtheorem{Lemma}[Definition]{Lemma}
\newtheorem{Proposition}[Definition]{Proposition}
\newtheorem{Corollary}[Definition]{Corollary}
\newtheorem{Example}[Definition]{Example}
\newtheorem{Remark}[Definition]{Remark}
\newtheorem{Note}[Definition]{Note}
\newcommand{\be}{\begin{equation}}
\newcommand{\ee}{\end{equation}}

\begin{document}

\title{\bf A Study of S-primary Ideals in Commutative Semirings}
\author{\bf Amaresh Mahato, Sampad Das, Manasi Mandal \footnote {e-mail : amaresh.blg2015@gmail.com, jumathsampad@gmail.com, manasi$_{-}$ju@yahoo.in } \ \\
{\small Department of Mathematics, Jadavpur University}\\
{\small 188, Raja S. C. Mallick Road, Kolkata -- 700032, India.}}
\date{}
\maketitle

\begin{abstract}
In this article, we define the concept of an $S$-$k$-irreducible ideal and $S$-$k$-maximal ideal in a commutative semiring. We also establish several results concerning $S$-$k$-primary ideals and prove the existence theorem and the $S$-version of the uniqueness theorem using localization, for $S$-$k$-primary decompositions.
 Also we show that the $S$-radical of every $S$-primary ideal is a prime ideal of $R$. Moreover, we investigate the structure of $S$-primary ideals in principal ideal semidomain and prove that each such ideal can be expressed of the form, $I = (vp^n)$, $n\in \mathbf{N}$ and for some $p \in \mathbf P -\mathbf P_S$ and $v\in R$ such that $(v)\cap S\neq \varnothing $, where  $\mathbf P$ is the set of all irreducible (prime) elements of R and for a multiplicative subset $S\subsetneq R$, the set $\mathbf P_S$ defined by $\mathbf P_S=\{p\in \mathbf P : (p) \cap S \neq \varnothing \}$.

\vspace{1cm}


\end{abstract}
\section{Introduction}

Let $R$ be a commutative semiring with unity. A non-empty subset $S \subsetneq R$ is called a multiplicative subset of $R$ if $1 \in S$ and $s_1 s_2 \in S$ for all $s_1, s_2 \in S$. An ideal $P$ of $R$ is said to be prime if $P\neq R$ and whenever $a,b\in R$ with $ab\in P$, then either $a\in P$ or $b\in P$.
Expanding on this idea, Hamed and Malek \cite{hm} introduced the concept of $S$-prime ideals as a generalization of prime ideals. Given a commutative ring $R$, a multiplicative subset $S\subsetneq R$ and an ideal $I\subseteq R$ such that $I\cap S=\phi$,  the ideal $I$ is called an $S$-prime ideal if there exists an element $s\in S$ such that for all $a,b\in R$, whenever $ab\in I$, it follows that $sa\in I$ or $sb\in I$.
Notably, when $S$ consists only of units in $R$, the concept of $S$-prime ideals reduces to that of classical prime ideals.
Separately, recall that an ideal $P\subseteq R$ is primary if for all $a,b\in R$, the inclusion $ab\in P$ implies that $a\in P$ or $b\in \sqrt{P}$, where $\sqrt{P}$ denotes the radical ideal of $P$ define as $\sqrt{P}=\{x \in R : x^n\in P , \text { where }n \in \mathbb{N}\}$.
Inspired by this and building on the generalization of prime ideals, Massaoud \cite{ms} introduced $S$-primary ideals in commutative rings, which extend the idea of primary ideals in a meaningful and inclusive way. Let $R$ be a commutative ring, $S$ a multiplicative subset of $R$, and $P$ a proper ideal of $R$ with $P\cap S=\phi$. Then $P$ is defined to be an $S$-primary ideal if there exists an element $s\in S$ such that for all $a,b\in R$, if $ab\in P$, then $sa\in P$ or $sb\in \sqrt{P}$. Several important results concerning $S$-primary ideals can also be found in Visweswaran \cite{vs}.
The absence of subtraction in semirings implies that many classical results from ring theory cannot be directly extended to semiring theory. To address this limitation, Henriksen~\cite{hn} introduced the notion of $k$-ideals. 
Specifically, if $R$ is a semiring, an ideal $I$ of $R$ is called a $k$-ideal if, for all 
$a,b \in R$, the conditions
$a+b \in I $ and $ a \in I$
imply that $b \in I$. This concept serves to compensate for the absence of additive inverses in semiring theory.
 After that numerous researchers have investigated $k$-ideals and their applications within semiring theory. As observed by Golan~\cite{go}, in a semiring, an ideal generated by a single element need not be a $k$-ideal. For instance, consider the ideal $I = \langle 1 + x \rangle$ in the semiring $\mathbb{N}[X]$, the semiring of polynomials with non-negative integer coefficients in the indeterminate $x$. We have
$(1 + x)^3 = (x^3 + 1) + 3x(1 + x) \in I$,
yet $x^3 + 1 \notin I$.
This clarify that $I$ is not a $k$-ideal in $\mathbb{N}[X]$. However, by imposing appropriate conditions on the semiring, it is possible to ensure that ideals generated by a single element are $k$-ideals. We see in \cite{sk}, (lemma 2.6),
 $R$ be additively cancellative, yoked and zerosumfree semiring, then for any ideal $I=<a> , a\in R $ is a $k$-ideal of $R$.

 A foundational contribution in Noetherian ring given by Noether \cite{no}, which laid the groundwork for the concept. During the past several decades, numerous generalizations of Noetherian rings have been actively explored by various authors due to its importance, one of the most notable is the introduction of $S$-Noetherian rings by Anderson and Dumitrescu \cite{ad}. The notion of $S$-Noetherianity, extensively studied in ring theory by several authors~\cite{aks, aks1, brt, as}, has been extended to semirings in order to develop a corresponding controlled finiteness condition. The notions of $S$-finite and $S$-stationary ideals in the context of $S$-Noetherian rings are also discussed in \cite{as}.
 In \cite{skb}, Kar, Bhoumik, and Goswami generalized the concept of $S$-prime ideals to $S$-$k$-prime and $S$-$k$-semiprime ideals. 
Moreover, Bhoumik and Goswami \cite{bs}, further extended the notion of $S$-primary ideals to $S$-$k$-primary ideals using the definition,  an ideal $A$ with $S\bigcap A \neq \varnothing$ of a semiring $R$ said to be $S$-$k$-primary ideal if $A$ is a $k$-ideal and for $a,b \in R $ if $ab \in A$ implies either $sa \in A$ or $sb \in \sqrt A$, for some $s \in S$. We also note that the $S$-version of the existence and uniqueness theorems for $S$-primary decomposition in $S$-noetherian ring theory is established in \cite{aks}. In \cite{sk}, the authors also investigate the primary decomposition of $k$-ideals within semiring theory.
In \cite{aq}, the authors study the structure of $S$-prime ideals and provide a characterization of such ideals in principal ideal domains.

In this paper, we extend certain important theorems and results related to $S$-$k$-primary ideals that are simultaneously $S$-primary and $k$-ideals and investigate their properties in the context of semirings.
 Furthermore, we define the notion of an $S$-$k$-irreducible ideal and show that, under certain conditions, every $S$-$k$-irreducible ideal is an $S$-$k$-primary ideal within the framework of $S$-Noetherian semirings. We also establish the existence theorem for $S$-$k$-primary decompositions in semirings. Moreover, by using the localization of ideals, we establish the $S$-version of uniqueness theorem for $S$-$k$-primary decompositions in semirings.
We establish some important results concerning the $S$-radical in a semiring $R$. In particular, the $S$-radical of an $S$-primary ideal of $R$ is a prime ideal. The $S$-radical of an $S$-prime ideal $P$ coincides with the $S$-radical of any $S$-$P$-primary ideal.
In this work, we examine $S$-primary ideals within principal ideal semidomain (PISD). 
We characterize the general form of such ideals and derive several fundamental structural results. 
Notably, we show that the radical of an $S$-primary ideal in a PISD is an $S$-maximal ideal.

   \subsection{S-k primary decomposition in commutative semirings}
 Massaoud \cite{ms} noted that an $S$-primary ideal $A$ of a ring $R$ is termed $S$-$P$-primary when its radical coincides with the $S$-prime ideal $P$; equivalently, when $\sqrt{A}=P$.

Similarly, we can define $S$-$k$-$P$-primary ideal, if $A$ is an $S$-$k$-primary ideal of a semiring $R$ and its radical satisfies 
$\sqrt {A} = P$, where $P$ is an $S$-prime ideal, then $A$ is called an $S$-$k$-$P$-primary ideal of $R$.

\begin{Proposition}
  The finite intersection of $S$-$k$-$P$-primary ideals is itself an $S$-$k$-$P$-primary ideal.
\end{Proposition}
\begin{proof}
   Let $A_1, A_2, \ldots, A_n$ be $S$-$k$-primary ideals of $R$ such that $\sqrt{A_i}=P$ and $A_i \cap S=\varnothing$ for each $i=1,2,\ldots,n$. Then $S \cap (\bigcap_{i=1}^n A_i)=\varnothing$. Let $A=\bigcap_{i=1}^n A_i$. Since each $A_i$ is $P$-primary, we obtain $\sqrt{A}=\sqrt{\bigcap_{i=1}^n A_i}=\bigcap_{i=1}^n \sqrt{A_i}=P$.
Now assume $ab \in A$ and $sb \notin A$ for all $s \in S$. Then there exists an index $k$ with $1 \le k \le n$ such that $ab \in A_k$ and $sb \notin A_k$ for every $s \in S$. Since $A_k$ is $S$-$P$-primary, there exists some $s \in S$ such that $sa \in \sqrt{A_k}=P$. Hence $sa \in \sqrt{A}=P$, and therefore $A$ is an $S$-$P$-primary ideal of $R$.
Finally, since the intersection of finitely many $k$-ideals of $R$ is again a $k$-ideal, it follows that $A=\bigcap_{i=1}^n A_i$ is a $k$-ideal of $R$.

This completes the proof of the theorem.

    \end{proof}
\begin{Definition}
     An $k$-ideal $A$ (disjoint with $S$) of a  semiring $R$ said to be $S$-$k$-irreducible ideal if  $s(I \cap J) \subseteq A \subseteq(I \cap J)$ for some $s \in S$ and for some $k$-ideals $I,J$ of $R$, then there exists $s' \in S$ such that  $ss'I \subseteq A$ or $ss' J\subseteq A$.
\end{Definition}
\begin{Example}

Let $R = \mathbb{Z}_0^+$ and consider the ideals $I = \langle 2 \rangle$ and $J = \langle 3 \rangle$. 
Then $I \cap J = \langle 6 \rangle$. 
Let $A = \langle 12 \rangle$ and $S = \{ 2^n : n \in \mathbb{N} \}$. 
It is evident that there exists some $s \in S$ such that
$s (I \cap J) \subseteq A \subseteq (I \cap J)$.
Since $A \cap S = \varnothing$, there exists $s' \in S$ such that either
$ss'I \subseteq A $ or $ ss'J \subseteq A$.
Therefore, although $A$ is not an irreducible ideal, it is $S$-$k$-irreducible.

\end{Example}



 \begin{Definition}
An ideal $I$ of $R$ (with $I \cap S = \emptyset$) is called \emph{$S$-finite} if there exist
$s \in S$ and a finitely generated ideal $J \subseteq R$ such that
$sI \subseteq J \subseteq I$.
A semiring $R$ is said to be \emph{$S$-Noetherian} if every ideal of $R$ disjoint from $S$
is $S$-finite.
 \end{Definition}
    



In a manner similar to the proofs presented in ring theory \cite{brt, as}, the lemma can also be verified within semiring theory.
   \begin{Lemma} \label{1}

Let $R$ be a semiring and $S$ a multiplicative subset of $R$. Then the following conditions are equivalent:

    (1) $R$ is an $S$-Noetherian semiring;
    
    (2) every ascending chain of ideals of $R$ is $S$-stationary;
    
    (3) every nonempty collection of ideals of $R$ contains an $S$-maximal element.

\end{Lemma}

\begin{proof}
\textbf{(1) $\Rightarrow$ (2):}
Suppose $R$ is $S$-Noetherian, and let 
$I_1 \subseteq I_2 \subseteq I_3 \subseteq \cdots$, be an ascending chain of ideals of $R$. Let $I = \bigcup_{n \ge 1} I_n$. Since $R$ is $S$-Noetherian, $I$ is $S$-finite; that is, there exist $s \in S$ and a finitely generated ideal $J$ such that 
$sI \subseteq J \subseteq I.$
 $J$ is finitely generated and $J \subseteq I$, there exists $k$ such that $J \subseteq I_k$. Hence for all $n \ge k$,
$sI_n \subseteq sI \subseteq J \subseteq I_k$.
Thus the chain $\{I_n\}$ is $S$-stationary.

\smallskip
\textbf{(2) $\Rightarrow$ (3):}
Assume every ascending chain of ideals is $S$-stationary. Let $\mathcal{F}$ be a nonempty family of ideals of $R$. Suppose, for contradiction, that $\mathcal{F}$ has no $S$-maximal element. Then we can construct an ascending chain
$I_1 \subsetneq I_2 \subsetneq I_3 \subsetneq \cdots, \quad I_n \in \mathcal{F}$,
such that for each $n\in\mathbb{N} $, and every $s \in S$, $sI_{n+1} \not\subseteq I_n$ (since $I_n$ is not $S$-maximal). 
By hypothesis, this chain is $S$-stationary, so there exist $k$ and $t \in S$ such that $tI_n \subseteq I_k$ for all $n \ge k$. In particular, $tI_{k+1} \subseteq I_k$, contradicting the choice of $I_{k+1}$. Hence $\mathcal{F}$ must contain an $S$-maximal element.

\textbf{(3) $\Rightarrow$ (1):}
Assume every nonempty family of ideals has an $S$-maximal element. Let $I$ be an ideal of $R$, and consider the collection
$\mathcal{G} = \{\, J \subseteq I \mid J \text{ is a finitely generated ideal of } R \,\}$.
By (3), $\mathcal{G}$ has an $S$-maximal element, say $J$. We claim that $I$ is $S$-finite. If not, then for every $s \in S$, $sI \not\subseteq J$, so there exists $x_s \in I$ such that $s x_s \notin J$. Pick one such $x \in I \setminus J$ and consider the ideal $J' = J + (x)$. Then $J'$ is finitely generated and $J \subsetneq J'$, but for all $s \in S$, $sJ' \not\subseteq J$ (since $s x \notin J$), contradicting the $S$-maximality of $J$. Therefore, there exists $s \in S$ such that $sI \subseteq J \subseteq I$. Hence $I$ is $S$-finite.
Since every ideal of $R$ is $S$-finite, $R$ is $S$-Noetherian.
\end{proof}
In \cite[Theorem~2.5]{sk}, it is shown that every $k$-irreducible ideal of an additively cancellative, yoked, and commutative Noetherian semiring is a primary ideal. 
In the present work, we extend this result by establishing its $S$-version.

\begin{Theorem} \label{2}
Let $R$ be an additive cancellative, zerosumfree, yoked, commutative $S$-Noetherian semiring. Then every $S$-$k$-irreducible ideal of $R$ is an $S$-$k$ primary ideal.
\end{Theorem}
\begin{proof}
Let $R$ be an additive cancellative, zerosumfree, yoked, and commutative $S$-Noetherian semiring, and let $A$ be an $S$-$k$-irreducible ideal of $R$. 
Suppose $a,b \in R$ satisfy $ab \in A$, and assume that $sb \notin A$ for every $s \in S$, where $S$ is a multiplicatively closed subset of $R$ with $S \cap A = \varnothing$. 
Under these conditions, we assert that there exists an element $r \in S$ such that $ra \in \sqrt{A}$.

  \
  
 We define two ideals of $R$ as
$I = A + \langle a^{k} \rangle$ and $J = A + \langle b \rangle$.
It is immediate that
$A \subseteq I \cap J$. Let us consider
$B_n=\{\,x \in R : a^{n}x \in A\,\}, $ $ n \in \mathbb{N}$.
For each $n$, the set $B_n$ is an ideal of $R$, and the sequence
$B_1 \subseteq B_2 \subseteq B_3 \subseteq \cdots$
is a non-decreasing chain of ideals. Since $R$ is $S$-Noetherian, Lemma \autoref{1}, ensures that this chain becomes $S$-stationary. Thus, there exist an index $k \in \mathbb{N}$ and an element $s \in S$ such that
$sB_n \subseteq B_k $ for every $ n \ge k$.
Let $y \in I \cap J$. Then we may write $y = a_0 + a^{k} z$ for some $z \in R$ and $a_0 \in A$. Since $ab \in A$, we obtain $aJ = aA + a\langle b\rangle \subseteq A$. Because $y \in J$, multiplying by $a$ gives $ay = aa_0 + a^{k+1} z \in A$. As $A$ is a $k$-ideal of $R$, the condition $ay \in A$ implies $a^{k+1} z \in A$, and hence $z \in B_{k+1}$. Therefore, $sz \in sB_{k+1} \subseteq B_k$, and Consequently, $a^{k} sz \in A$. It follows that $sy =  sa_0 +a^{k} sz  \in A$.
From the above, we obtain $s(I \cap J) \subseteq A \subseteq I \cap J$. Let $\bar{A}$ denote the $k$-closure of an ideal $A$, defined by $\overline{A} = \{\, a \in R : a + b = c \text{ for some } b, c \in A \,\}$. A standard property of the $k$-closure states that $\overline{(I \cap J)} \subseteq \overline{I} \cap \overline{J}$. Since $A \subseteq I \cap J$ and $A$ is a $k$-ideal, it follows that $A = \overline{A} \subseteq \overline{I} \cap \overline{J}$. Therefore $s(\overline I \cap \overline J) \subseteq \overline{A} \subseteq \overline{I} \cap \overline{J}$.
Since $A$ is $S$-$k$-irreducible, there exists $s'\in S$ such that either
$ss'\overline{I} \subseteq A $ or $ ss'\overline{J} \subseteq A.$
If $ss'\overline{J} \subseteq A$, then $ss'b \in A$, which is impossible. Therefore,
$ss'\overline{I} \subseteq A.$
Hence $ss'a^{k} \in A$ for some $k\in\mathbb{N}$. Let $r = ss' \in S$. Then $ra^{k}\in A$ implies
$(ra)^{k} = r^{\,k-1}(ra^{k}) \in A$.
Thus,
$ra \in \sqrt{A}.$

  \end{proof}
 As observed in \cite{bs1}, the notion of an $S$-maximal ideal in a semiring was introduced as a natural extension of the classical concept of maximal ideals.  Moreover, it is shown in \cite{ea} that $R$ possesses at least one $k$-maximal ideal.

  Now we define $S$-$k$-maximal ideal in a  semiring $R$.

  \begin{Definition}
     Let $R$ be a semiring, $S\subsetneq R$ be a multiplicative set, and $M$ (disjoint with $S$) be a $k$-ideal of $R$. Then $M$ is said to be $S$-$k$-maximal if there exists $s\in S$ such that whenever $M\subseteq I$ for some ideal $I$, then either $sI\subseteq M$ or $I\cap S \neq \varnothing$. 
  \end{Definition}

\begin{Theorem}  (Existence of $S$-$k$-primary decomposition)
    Let $R$ be an $S$-Noetherian semiring. Then every proper $k$-ideal (disjoint with $S$) of $R$ can be written as a finite intersection of $S$-$k$-primary ideals.
\end{Theorem}

\begin{proof}

Let $E$ denote the family of all $k$-ideals of $R$ which are disjoint from $S$ and which cannot be represented as a finite intersection of $S$-$k$-irreducible ideals. Our aim to prove that $E=\varnothing$.
Assume, contrary to our claim, that $E\neq\varnothing$. Since $R$ is an $S$-Noetherian semiring, by Lemma \autoref{1}, ensure that $E$ possesses an $S$-$k$-maximal element; denote this ideal by $I$. By the defining property of $E$, the ideal $I$ is not $S$-$k$-irreducible, and thus it is not irreducible as a $k$-ideal.
Consequently, there exist $k$-ideals $J$ and $K$ of $R$ such that
$I = J \cap K,$
with $I\neq J$ and $I\neq K$. 
Since $I \subseteq J \cap K$, it follows that $s(J \cap K) \subseteq I \subseteq J \cap K$ for every $s \in S$. Because $R$ is an $S$-Noetherian semiring, all ideals of $R$ are $S$-finite. Moreover, $I$ is not $S$-$k$-irreducible; therefore, for any $s, s' \in S$, we must have $ss'J \nsubseteq I$ and $ss'K \nsubseteq I$.
We now argue that neither $J$ nor $K$ lies in the set $E$. Assume the opposite: suppose both $J$ and $K$ belong to $E$. Since $I$ is an $S$-$k$-maximal element of $E$, there exist some $s, s' \in S$ such that $sJ \subseteq I$ and $s'K \subseteq I$, while $I \subsetneq J$ and $I \subsetneq K$. This contradicts the maximality of $I$ in the set $E$.
 Thus we conclude that $J, K \notin E$.
Consequently, both $J$ and $K$ admit finite decompositions into $S$-$k$-irreducible ideals. Since $I = J \cap K$, it follows that $I$ itself can be expressed as a finite intersection of $S$-$k$-irreducible ideals. This contradicts the assumption that $I$ lies in $E$.
Therefore, $E$ must be empty. Hence every proper $k$-ideal of $R$ has a finite decomposition into $S$-$k$-irreducible ideals. By Theorem~\autoref{2}, every $S$-$k$-irreducible ideal is $S$-$k$-primary. Consequently, any $k$-ideal of $R$ can be expressed as a finite intersection of 
$S$-$k$-primary ideals.

 \end{proof}
 The following result arises as a direct consequence of this theorem.
 
\begin{Corollary}
If $R$ is an $S$-Noetherian semiring, then radical of $k$-ideal $I$ of $R$ with $I \cap S = \varnothing$ can be represented as the finite intersection of $S$-prime ideals.
\end{Corollary}

\begin{proof}
Let $R$ be an $S$-Noetherian semiring and a $k$-ideal $I$ disjoint with $S$.  
Since $R$ is $S$-Noetherian, every $k$-ideal of $R$ admits a finite $S$-$k$-primary decomposition. 
Hence, there exist $S$-$k$-primary ideals $A_1, A_2, \ldots, A_r$ such that 
$I = A_1 \cap A_2 \cap \cdots \cap A_r$.
Now, since $I$ is a radical ideal, we have $I = \sqrt I$. Taking radicals on both sides gives $I = \sqrt I= \sqrt A_1 \cap \sqrt A_2 \cap \cdots \cap \sqrt A_r$.
As the radical of an $S$-primary ideal is $S$-prime, each $\sqrt A_i$ is an $S$-prime ideal. Let $P_i =\sqrt A_i$ for $i = 1,2,\ldots,r$. Thus, \[I = P_1 \cap P_2 \cap \cdots \cap P_r,\] which represents $I$ as a finite intersection of $S$-prime ideals.
\end{proof}
    
\begin{Remark} (\cite{bs})
    Radical of a $S$-$k$-primary ideal may not be $S$-$k$-prime ideal in a semiring $R$.
\end{Remark} 

Recall that if $R$ is a commutative semiring and $I$ is an ideal of $R$, then for any 
$s \in R$ define
$I : s = \{\, x \in R : sx \in I \,\}$.
For each $s \in R$, the set $I : s$ is an ideal of $R$.

\begin{Proposition}
      
Let $f : R \to R'$ be an homomorphism of commutative semirings such that $f(S)$ does not contain
zero. 
If $A$ is an $f(S)$-$k$-primary ideal of $R'$, 
then $(A:f(s))$ is a $k$-primary ideal of $R'$ and $f(f^{-1}(A):s) \subseteq (A:f(s))$, where $s\in S$.

\end{Proposition}
\begin{proof}
Since $A$ is an $f(S)$-$k$-primary ideal of $R'$, it follows that $A$ is a $k$-primary ideal. 
Hence, by [\cite{bs}, Proposition 2.9], there exists $s \in S$ such that $(A : f(s))$ is a $k$-primary ideal of $R'$.

Let $x \in (f^{-1}(A) : s)$. It follows that 
$f(x) \in  f(f^{-1}(A) : s)$. Again $x \in (f^{-1}(A) : s)$. Then $x s \in f^{-1}(A)$ for some $s \in S$. 
Applying $f$ to both sides, we obtain $f(x s) \in f(f^{-1}(A)) = A$. 
This implies that $f(x) f(s) \in A$, and consequently $f(x) \in (A : f(s))$.
Therefore,
$f(f^{-1}(A) : s) \subseteq f(f^{-1}(A)) : f(s)$.
Moreover, equality holds when $f$ is the identity mapping, that is, when $f(x) = x$.
Hence
$f(f^{-1}(A) : s) \subseteq f(f^{-1}(A)) : f(s) = (A : f(s))$.

\end{proof}
\begin{Note}

If $f : R \to R'$ is isomorphism then,  $f(f^{-1}(A):s)= f(f^{-1}(A)):f(s))$.

Let $y \in (A : f(s))$. Then there exists $y_1 \in R$ such that $f(y_1) = y$.  
Since $f(y_1) \in (A : f(s))$, we have 
$f(y_1) f(s) \in A$. 
Hence
$f(y_1 s) \in A = f(f^{-1}(A))$. 
This shows that $y_1 s \in f^{-1}(A)$. Therefore, $y_1 \in (f^{-1}(A) : s)$, which implies $f(y_1) \in f(f^{-1}(A) : s)$.
Thus, $y \in f(f^{-1}(A) : s)$.  
Consequently, we obtain
$f(f^{-1}(A)) : f(s) \subseteq f(f^{-1}(A) : s)$
and so, we get   $f(f^{-1}(A):s)= f(f^{-1}(A)):f(s))$.
\end{Note}
\begin{Proposition}

  Let $A$ be an $S$-$k$-primary ideal of $R$, and let $x \in R$ satisfy 
$sx \notin A$ for every $s \in S$.
 Then $(A:sx)$ is a $S$-$k$-primary ideal of $R$, and there exists $r\in S $ such that, $r\sqrt{(A:sx)}\subseteq\sqrt{A} \subseteq\sqrt{(A:sx)}$,  for all $s\in S$.
\end{Proposition}
\begin{proof}
Assume first that $(A:sx) \cap S \neq \varnothing$ for some $s \in S$. Then there exists $r \in S$ with $r \in (A:sx)$, which means $rsx \in A$. This contradicts the fact that $A$ is an $S$-primary ideal and $sx \notin A$. Hence $(A:sx) \cap S = \varnothing$ for every $s \in S$.
Now take $s \in S$ and suppose $ab \in (A:sx)$ with the additional assumption that $s'a \notin \sqrt{(A:sx)}$ for all $s' \in S$. From $ab \in (A:sx)$, we obtain $sxab \in A$. Since no element of the form $s'a$ lies in $\sqrt{(A:sx)}$, the $S$-primary property of $A$ ensures that there exists $s'' \in S$ such that $s''sxb \in A$, and consequently $s''b \in (A:sx)$. Thus $(A:sx)$ is an $S$-primary ideal.
Next we verify that $(A:sx)$ is a $k$-ideal. Clearly $(A:sx) \subseteq \overline{(A:sx)}$. Let $a \in \overline{(A:sx)}$. By definition of the $k$-closure, there exists $b \in (A:sx)$ such that $a+b \in (A:sx)$. This means $sx(a+b) \in A$ with $sxb \in A$. Hence $sxa + sxb \in A$, and since $A$ is a $k$-ideal, it follows that $sxa \in A$. Therefore $a \in (A:sx)$, showing that $(A:sx)=\overline{(A:sx)}$, and hence $(A:sx)$ is a $k$-ideal of $R$.
Finally, let $c \in \sqrt{(A:sx)}$. Then $c^{n} \in (A:sx)$ for some $n \in \mathbb{N}$, which gives $sxc^{n} \in A$. Since $A$ is $S$-$k$-primary and $rsx \notin A$ for all $r \in S$, there exists an element $r \in S$ such that $(rc^{n})^{k} \in A$ for some $k \in \mathbb{N}$. Thus $(rc)^{nk} = r^{nk-k}(rc^{n})^{k} \in A$, showing that $rc \in \sqrt{A}$. Therefore $r\,\sqrt{(A:sx)} \subseteq \sqrt{A}$. Since $A \subseteq (A:sx)$, we also have $\sqrt{A} \subseteq \sqrt{(A:sx)}$, and hence $r\,\sqrt{(A:sx)} \subseteq \sqrt{A} \subseteq \sqrt{(A:sx)}$.

\end{proof}
Localization plays a central role in commutative algebra. Although localization techniques do not carry over directly to the setting of commutative semirings, several important aspects remain effective and applicable. In this section, we consider semirings obtained through localization and show how these constructions can be used to establish the uniqueness theorem for decomposition.
Recall from \cite{nas}, 
$f : R \longrightarrow S^{-1}R$
be the canonical semiring homomorphism defined by $f(x) = \frac{x}{1}$.  
For any ideal $I$ of $R$, the preimage $f^{-1}(S^{-1}I)$ is called the contraction of $I$ with respect to $S$, and it is denoted by
$S(I)=\{\, x \in R \mid \tfrac{x}{1} \in S^{-1}I \,\}.$
Clearly, $I \subseteq S(I)$. Therefore, to investigate ideal decompositions compatible with localization, we require an $S$-version of primary decomposition of ideals in  additive cancellative semiring, which generalizes the classical concept of primary decomposition.

Now we define S-version of minimal decomposition in an additive cancellative, zerosumfree, yoked, commutative semiring R. Let $S \subsetneq R$ is a multiplicatively
closed subset.
Consider a $k$-ideal $I \subseteq R$ satisfying $I \cap S = \varnothing$. We say that $I$ possesses an $S$-$k$-primary decomposition if it can be expressed as the intersection of
finitely many $S$-$k$-primary ideals of $R$. Whenever such a representation
exists, the ideal $I$ is called $S$-$k$-decomposable.
Suppose that
$I=\bigcap_{i=1}^{n} A_i$, for  $i \in\{1,2,\ldots,n\}$,
where each $A_i$ is an $S$-$k$-primary ideal and $\sqrt{(A_i)}=P_i$.
This decomposition is called \emph{minimal} if the following conditions hold:

 (1) $S(P_i)\neq S(P_j)$ for all distinct $i,j\in\{1,2,\ldots,n\}$;
 
 (2) For each $i\in\{1,2,\ldots,n\}$,
 $S(A_i)\not\subseteq S(A_j) \quad \text{for every } j\neq i$,
     or equivalently, $S(A_i)\not\subseteq \bigcap_{j\neq i} S(A_j)$.

In the special case $S = \{1\}$, the definition of an $S$-primary ideal coincide with an ordinary primary ideal. Consequently, $S$-primary decomposition becomes identical to the usual primary decomposition.

\begin{Lemma}\label{6}
If $A$ is $S$-$P$-primary, then $S(A)$ is $S(P)$-primary in the semiring $R$.
\end{Lemma}
\begin{proof}
    Suppose $a,b \in R$ with $ab \in S(A)$. Then $\frac{ab}{1} \in S^{-1}A$, so there exist
$x \in A$ and $s \in S$ such that $\frac{ab}{1} = \frac{x}{s}$. Hence, for some $s_1 \in S$,
$s_1sab \in A$. Since $A$ is $S$-$P$-primary, there exists $s' \in S$ such that
$s's_1sa \in A$ or $s'b \in \sqrt{A}=P$. Therefore, either $\frac{a}{1} =\frac{s's_1sa}{s's_1s}\in S^{-1}A$ or
$\frac{b}{1}=\frac{s'b}{s'} \in S^{-1}P$, which shows that $a \in S(A)$ or $b \in S(P)$. So $S(A)$ is
$S(P)$-primary.

\end{proof}
\begin{Theorem}
    Let $R$ be an additively cancellative, yoked and commutative $S$-Noetherian semiring, and let
$I = \bigcap_{i=1}^{n} A_i$
be a minimal $S$-$k$-primary decomposition of $I$, where each $A_i$ is $P_i$-$S$-$k$-primary for $i = 1,2,\ldots,n$. Then the ideals $S(P_i)$ coincide with the prime ideals appearing in the set $\{ S (\sqrt{(I : x)} \} $, $ x \in R$.
Consequently, the ideals $S(P_i)$ do not depend on the specific choice of an $S$-primary decomposition of $I$.

\end{Theorem}
\begin{proof}

Let $x \in R$, the colon ideal decomposes as $(I : x) = (A_1 : x) \cap \cdots \cap (A_n : x)$, by Lemma 2.8(iv) of \cite{sk}, whenever $x \notin A_i$, we have $\sqrt{(A_i : x)} = P_i$. Thus $\sqrt{(I : x)} = \bigcap_{{i \mid x \notin A_i,}} P_i$.
Localization preserves finite intersections, so $S^{-1}(I : x) = \bigcap_{i=1}^n (S^{-1}A_i : S^{-1}R \frac{x}{1})$. Contracting to $R$ gives $S(I : x) = \bigcap_{i=1}^n (S(A_i) : x)$. Hence $S(\sqrt{(I : x)}) = \bigcap_{{i \mid x \notin A_i,}} \sqrt{S(A_i) : x}$.
Each ideal $S(P_i)$ is prime, since $S^{-1}P_i$ is prime and the contraction of a prime ideal remains prime. By Lemma \autoref{6}, the ideal $S(A_i)$ is $S(P_i)$-primary. By Lemma 2.8(iv) of \cite{sk}, therefore yields $S(\sqrt{(I : x)}) = \bigcap_{{i \mid x \notin A_i,}} S(P_i)$.
Suppose $S(\sqrt{(I : x)})$ is prime. Being an intersection of the prime ideals $S(P_i)$, ( \cite{am}, Proposition 1.11), implies that $S(\sqrt{(I : x)}) = S(P_j)$ for some $j$. Thus, any prime ideal arising as $S(\sqrt{(I : x)})$  must be equal to some $S(P_j)$.
Conversely, the minimality of the reduced $S$-primary decomposition ensures that, for each $i$, there exists an element $x_i \in (\bigcap_{j \ne i} S(A_j)) \setminus S(A_i)$. For such $x_i$, we have $S(I : x_i) = \bigcap_{j=1}^n (S(A_j) : x_i) = (S(A_i) : x_i)$, because $(S(A_j) : x_i) = R$ for every $j \ne i$. Applying Lemma 2.8(iv) of \cite{sk}, we obtain $S(\sqrt{(I : x_i)}) = S(P_i)$ for $1 \le i \le n$.

Thus each $S(P_i)$ arises as $S(\sqrt{(I : x)})$ for some $x \in R$, and no other prime ideal of this form can occur. This proves the uniqueness of the set of associated prime radicals and completes the proof.
\end{proof}

\section {S-radical on S-primary ideal}
\begin{Definition}
    Let $R$ be a commutative semiring and $S$ be a multiplicative subset of
 $R$. The $S$-radical of an ideal $I$ is defined by 

        $\sqrt[S]{I}=\{a\in R : sa^n \in I $, for some $s\in S$ and $n\in \mathbf{N}$\} 
\end{Definition}
\begin{Example}
Let $R=\mathbf{Z}_0^{+}$ and let 
$S=\{\,2^{n} : n \in \mathbb{N}\,\}$
be a multiplicatively closed subset of $R$. Consider the ideal $I=(18)$ of $R$. 
Then the $S$-radical of $I$ is given by
$\sqrt[S]{I}
=\{\,a \in R : s a^{n} \in (18) \text{ for some } s \in S \text{ and } n \in \mathbb{N}\,\}$.
From this, it follows that
$\sqrt[S]{I} = (3)$.

\end{Example}
    
    
    
    
   

\begin{Theorem} \label{3}
 $S$-radical of every  $S$-prime ideal is a prime ideal  of a semiring $R$.
\end{Theorem}
\begin{proof}
   Let $P$ be an $S$-prime ideal of $R$. We show that $\sqrt[S]{P}$ is a prime ideal of $R$. 
Take $x, y \in R$ with $xy \in \sqrt[S]{P}$. By the definition of the $S$-radical, 
there exist $s \in S$ and $n \in \mathbb{N}$ such that 
$s(xy)^n = s x^{n} y^{n} \in P$.
Since $P$ is $S$-prime, either there exists $s_{1} \in S$ such that $s_{1} x^{n} \in P$, in that case 
$x \in \sqrt[S]{P}$, or, if $s_{1} x^{n} \notin P$, the $S$-primeness of $P$ guarantees the existence of $s_{2} \in S$ such that 
$s_{2}(s y^{n}) \in P$. 
Setting $s_{3} = s_{2}s$, we obtain $s_{3} y^{n} \in P$, and hence $y \in \sqrt[S]{P}$.

Therefore, whenever $xy \in \sqrt[S]{P}$, it follows that 
$x \in \sqrt[S]{P}$ or $y \in \sqrt[S]{P}$. 
Thus, $\sqrt[S]{P}$ is a prime ideal of $R$.

\end{proof}
\begin{Theorem} \label{4}
   $ S$-radical of every $S$-primary ideal is a prime ideal of a semiring $R$.
\end{Theorem}
\begin{proof}
Let $A$ be an $S$-primary ideal of $R$. We show that $\sqrt[S]{A}$ is a prime ideal of $R$. 
Take $x, y \in R$ such that $xy \in \sqrt[S]{A}$. By the definition of the $S$-radical, 
there exist $s \in S$ and $n \in \mathbb{N}$ such that 
$s(xy)^{n} = sx^{n}y^{n} \in A$.
Since $A$ is $S$-primary, either there exists $s_{1} \in S$ with $s_{1}x^{n} \in A$, 
which implies $x \in \sqrt[S]{A}$, or, if $s_{1}x^{n} \notin A$, then the $S$$-$primary property of $A$ guarantees the existence of $s_{2} \in S$ and $k \in \mathbb{N}$ such that 
$(s_{2}s y^{n})^{k} \in A$.
Setting $s_{3} = (s_{2}s)^{k}$ gives $s_{3}y^{nk} \in A$, and hence $y \in \sqrt[S]{A}$.
Thus, whenever $xy \in \sqrt[S]{A}$, we have 
$x \in \sqrt[S]{A}$ or $y \in \sqrt[S]{A}$. 
Therefore, $\sqrt[S]{A}$ is a prime ideal of $R$, and consequently the $S$-radical of every $S$-primary ideal is a prime ideal of $R$.

\end{proof}

\begin{Remark}
    Let $P$ be a $S$-prime ideal, then $S$-radical of $P$ and $S$-$P$ primary ideal are same.
\end{Remark}

 \section {S-Primary ideals in PISD} 
 
Recall, 
   A semidomain $R$ is called a principal ideal semidomain if every ideal of $R$ can be expressed as a principal ideal. 
In other words, for any ideal $I \subseteq R$, there exists an element $a \in R$ such that 
$I = (a) = \{ ra : r \in R \}$.

\begin{Example}
   Let the semiring $R=(\mathbf{N}_0,gcd,\cdot)$ is a  principal ideal semidomain, since every ideal of $R$ is principal ideal generated by a single element $(a)=a\mathbf{N}_0$, where $a\in\mathbf{N}_0 $. 
\end{Example}

We adopt the following notations, consistent with the conventions used in \cite{aq}, and employ them throughout the part.
 Let $R$ be a principal ideal semidomain, and let $\mathbf{P}$ denote the set of all irreducible (prime) elements of $R$ define in \cite{na}.
 For a multiplicative subset $S$ of $R$, define
$\mathbf{P}_S = \{\, p \in \mathbf{P} : (p) \cap S \neq \varnothing \,\}$. Hence, $\mathbf{P}_S$ is the collection of irreducible elements of $R$ that divide some element in $S$.
 Equivalently, an irreducible element $p$ belongs to $\mathbf{P}_S$ if there exist $s \in S$ and $b \in R$ such that $s = bp$.

We will frequently use the following result, which appears in (\cite{bs}, proposition 2.9).
\begin{Proposition}
     Let $R$ be a commutative semiring, $S$ a multiplicative subset of $R$ and A an ideal of $R$ disjoint from $S$. Then the following assertions are equivalent:
     
(1) $A$ is an $S$-primary ideal of $R$.

(2)  ($A:s$) is a primary ideal of $R$ for some $s\in S$.
\end{Proposition}
The $S$-primary ideals of a principal ideal semidomain are completely determined in the following results.
\begin{Theorem}

    Let $R$ be a principal ideal semidomain and $S$ be a multiplicative subset
 of $R$ and let $A$ be an ideal of $R$. The following statements are equivalent:

 (1) $A$ is an $S$-primary ideal of $R$.
 
 (2)  $A = (vp^n)$, $n\in N$ and for some $p \in \mathbf P -\mathbf P_S$ and $v\in R$ such that $(v)\cap S\neq \varnothing $. 
\end{Theorem}
\begin{proof}

\textbf{(1) $\Rightarrow$ (2):}
 Let $A = (q)$ be a nonzero $S$-primary ideal of $R$, and let $s \in S$ be such that $(A : s)$ is a primary ideal of $R$.  
Since $A \subseteq (A : s)$, we have $(q) \subseteq (A : s)$. Hence, there exists an irreducible element $p \in R$ such that $(A : s) = (p^n)$ for some $n \in \mathbb{N}$. Consequently, $(q) \subseteq (p^n)$, which implies $p^n s \in A$. Thus, $p^n s = q' q$ for some $q' \in R$. In particular, $q' q \in (p^n)$, and therefore either $q' \in (p^n)$ or $q \in \sqrt{(p^n)}$.
If $q' \in (p^n)$, then $q' = q'' p^n$ for some $q'' \in R$. Hence,
$s p^n = q' q = q q'' p^n$,
which implies $s = q q'' \in (q) \cap S$. This contradicts the assumption that $(q)$ is $S$-primary and $(q) \cap S = \varnothing$. Hence $q \in \sqrt{(p^n)}$ and since $(q) \subseteq (p^n)$, there exists $v \in R$ and $m \in \mathbb{N}$ such that $q = v p^m$ with $m \geq n$.  
Suppose $m > n$. Then
$s p^n = q' q = v q' p^m$,
which implies $s = v q' p^k$ where $k + n = m$. Hence, $p \in \mathbf{P}_S$.  
If $p \in \mathbf{P}_S$, then $(p) \cap S \neq \varnothing$, which further implies $(p^n) \cap S \neq \varnothing$.  
Let $s' \in S$ such that $s' = c p^n$ for some $c \in R$. Then
$s s' = s c p^n \in I$,
contradicting the fact that $A = (q)$ is $S$-primary and $(A) \cap S = \varnothing$.
Therefore, we must have $m = n$, so that $q = v p^n$ for some $v \in R$.  
In this case,
$s p^n = q' q = v q' p^n$,
which implies $s = v q'$. Hence, $(v) \cap S \neq \varnothing$, and it follows that $A = (q) = (v p^n)$ with $(v) \cap S \neq \varnothing$ and $p \in \mathbf{P} \setminus \mathbf{P}_S$.

\textbf{(2) $\Rightarrow$ (1):} Let $A = (v p^n)$, where $p \in \mathbf{P} - \mathbf{P}_S$ and $(v) \cap S \neq \varnothing$.  
Then there exist elements $s \in S$ and $v_0 \in R$ such that $s = v v_0$.  
Let $x \in (A : s)$. Then $s x \in A$, and hence $s x = \alpha v p^n$ for some $\alpha \in R$.  
Therefore, $s x \in (p^n)$, which implies that $x \in (p^n)$, since $s \notin (p)$.  
Thus, $(A : s) \subseteq (p^n)$.
Conversely, since $p^n s = v v_0 p^n \in A$, it follows that $(p^n) \subseteq (A : s)$.  
Hence, $(A : s) = (p^n)$, which is a primary ideal of $R$.  
Therefore, $A$ is an $S$-primary ideal of $R$.

In particular, if $n = 1$, then $A = (v p)$, where $p \in \mathbf{P} - \mathbf{P}_S$ and $v \in R$ with $(v) \cap S \neq \varnothing$, is an $S$-prime ideal of the principal ideal semidomain $R$.


\end{proof}

\begin{Example}
  Consider the semiring $R = (\mathbf{Z}_0^+, \gcd, \cdot)$ and the multiplicative set 
$S = \{3^k : k \in \mathbf{N}\}$. 
In this case, the only irreducible element that divides an element of $S$ is $3$, and hence 
$\mathbf{P}_S = \{3\}$.
Let $A$ be a nonzero $S$-primary ideal of $\mathbf{Z}_0^+$. 
Such an ideal can be written in the form $A = (v p^k)$, where $p$ is a prime with $p \ne 3$ 
and $v \in \mathbf{Z}_0^+$ satisfies $(v) \cap S \ne \varnothing$. 
Thus, there exist $m \in \mathbf{Z}_0^+$ and $k \in \mathbf{N}$ such that $m v = 3^k$, 
which implies that $v = 3^l$ for some $l \in \mathbf{N}$.
Therefore, the $S$-primary ideals of $\mathbf{Z}_0^+$ are those of the form 
$A = (3^l p^k)$, where $p \ne 3$ is a positive prime integer and $l, k \in \mathbf{N}$.

\end{Example}

\begin{Theorem}
    
Let $R$ be a principal ideal semidomain. If $A = (v p^k)$ is an $S$-primary ideal of $R$, then its radical $P = \sqrt{A}$ is an $S$-maximal ideal of $R$.

\end{Theorem}
   
\begin{proof}

Let $A$ be a nonzero $S$-primary ideal of a principal ideal semidomain $R$. 
Then $A = (v p^k)$ for some $p \in \mathbf{P} - \mathbf{P}_S$ and $v \in R$ with $(v) \cap S \neq \varnothing$. 
Hence, the radical of $A$ is $P = \sqrt{A} = (v p)$ is a S-prime ideal. 
Let $I = (x)$ be an ideal of $R$ such that $P \subseteq I$. 
Since $v p \in (x)$, we have $v p = xy$ for some $y \in R$. 
In particular, $xy \in (p)$, which implies that either $x \in (p)$ or $y \in (p)$.
  
  \textbf{Case 1:} Suppose that $x \in (p)$. Then $x = x' p$ for some $x' \in R$. 
Substituting into $v p = xy$, we have $v p =x' y p$, and hence $v = x'y$. 
Since $(v) \cap S \neq \varnothing$, there exists $t \in R$ such that 
$s = t v = t x'y \in S$. 
Therefore,
$s x = t v x = t x' v p \in (v p)$.
It follows that $s I \subseteq P$.

 \textbf{Case 2:} Suppose that $x \notin (p)$. Then $y\in (p)$, so $y = y' p$ for some $y' \in R$. 
Substituting into $v p = xy$, we obtain $v p =x y' p $, and hence $v = xy' $. 
Since $(v) \cap S \neq \varnothing$, it follows that 
$\varnothing \neq (v) \cap S \subseteq (x) \cap S$.
Therefore, $I \cap S \neq \varnothing$, which contradicts the assumption that $A$ is an $S$-primary ideal. So,  $P=\sqrt{A}$ is a $S$-maximal ideal of $R$.

\end{proof}

\begin{Proposition}\label{5}

    In PISD, $S$-radical of every $S$-primary ideal is a prime ideal of $R$.
\end{Proposition}
\begin{proof}
  
Let $R$ be a principal ideal semidomain (PISD). 
Then every $S$-primary ideal of $R$ is of the form $I = (v p^k)$, 
where $k \in \mathbf{N}$, $p \in \mathbf{P} - \mathbf{P}_S$, and $v \in R$ such that $(v) \cap S \neq \varnothing$. 
Recall that the $S$-radical of an ideal $A$ is defined by
$\sqrt[S]{A} = \{\, a \in R : s a^n \in A \text{ for some } s \in S \text{ and } n \in \mathbf{N} \,\}$.
Therefore,
$\sqrt[S]{I} = \{\, p_1 \in R : s p_1^n \in I \text{ for some } s \in S \text{ and } n \in \mathbf{N} \,\}$.
It follows that $\sqrt[S]{I} = (p)$, which is a prime ideal of $R$.

\end{proof}
\begin{Proposition}
     In PISD, $S$-radical of nonzero $S$-primary ideal is a maximal ideal of $R$.
\end{Proposition}
\begin{proof}
   From Proposision \autoref{5}, we see that in a PISD, the $S$-radical of every $S$-primary ideal is a prime ideal of $R$. Moreover, in a PISD, every nonzero prime ideal is a maximal ideal of $R$.

\end{proof}

\begin{Proposition}
  Let $R$ be a in a PISD, $S$ be a multiplicative subset of $R$ and $A$ be an primary ideal disjoint with $S$ of $R$ . If  $P=\sqrt{A}$ is an $S$-$k$-prime ideal of $R$, then $A$ is also an $S$-$k$-primary ideal of $R$. We say $A$ is $S$-$k$-$P$-primary ideal of $R$.
\end{Proposition}
\begin{proof}
   Its a obvious case . 
\

 If $R$ is not a PISD (Principal Ideal Semidomain), the above result may not hold.
Consider the semiring $R = (\mathbf{Z}_0^+,+,\cdot)$, which is not a PISD.  
Let $P = 2\mathbf{Z}_0^+$ and $A = 2\mathbf{Z}_0^+ \setminus \{2\}$.
Then $\sqrt{A} = P$, so $P$ is a $k$-ideal, but $A$ is not a $k$-ideal.  
Indeed, $2 + 4 = 6 \in A$ and $4 \in A$, but $2 \notin A$, which violates the condition for being a $k$-ideal.
Moreover, $A$ is not a primary ideal, since $3 \cdot 2 = 6 \in A$, but $2 \notin A$ and $3 \notin \sqrt{A}$.
However, if we take 
$S = \{2n + 1 : n \in \mathbf{N}_0\}$,
then $A$ is an $S$-primary ideal.

\end{proof}

\end{document}